\documentclass[11pt]{article}

\usepackage[english]{babel}
 \usepackage[T1]{fontenc}
\usepackage{fancyhdr}
\usepackage{indentfirst}
\usepackage{graphicx}
\usepackage{newlfont}
\usepackage{cite}
\usepackage{lipsum}

\newcommand\blfootnote[1]{%
  \begingroup
  \renewcommand\thefootnote{}\footnote{#1}%
  \addtocounter{footnote}{-1}%
  \endgroup
}

\usepackage{amsfonts}
\usepackage{amssymb}
\usepackage{amsmath}
\usepackage{latexsym}
\usepackage{amsthm}

\theoremstyle{plain}                   
\newtheorem{theorem}{Theorem}
\newtheorem{lem}[theorem]{Lemma}
\newtheorem{cor}[theorem]{Corollary}

\theoremstyle{remark}
\newtheorem*{rem}{Remark}

\begin{document}

\title{Heisenberg Hausdorff dimension \\ of Besicovitch sets}
\author{Laura Venieri}
\date{}

\maketitle

\begin{abstract}
We consider (bounded) Besicovitch sets in the Heisenberg group and prove that $L^p$ estimates for the Kakeya maximal function imply lower bounds for their Heisenberg Hausdorff dimension.
\end{abstract}

\blfootnote{
The author is supported by the Academy of Finland.}

\section{Introduction}

In this paper we investigate \textit{Besicovitch sets} (also called Kakeya sets) in the context of the Heisenberg group, examining in particular their Heisenberg Hausdorff dimension. The Kakeya sets contain a unit line segment in every direction and have Lebesgue measure zero. Their name comes from a question that Kakeya posed in 1917: what is the smallest area in which a unit line segment can be rotated of 180 degrees in the plane? Besicovitch proved that this can be done in arbitrarily small area.

In recent years much research has been done about their (Euclidean) Hausdorff dimension. The \textit{Kakeya conjecture} states that every Besicovitch set in $\mathbb{R}^n$ must have Hausdorff dimension $n$. Davies \cite{Davies1971} proved that this holds in $\mathbb{R}^2$ but it is still an open problem in $\mathbb{R}^n$ for $n \ge 3$, even if these sets have been studied extensively because of their close connection to central questions in modern Fourier analysis.

Many different methods have been used to find lower bounds for the Hausdorff dimension: the first were based on geometric observations and developed in particular by Bourgain \cite{Bourgain1991} and Wolff \cite{MR1363209}. More recently Bourgain \cite{Bourgarithm} introduced an arithmetic combinatorics method obtaining the lower bound $\frac{13}{25}n + \frac{12}{25}$, then improved by Katz and Tao \cite{Katz&Tao1} to $\frac{6}{11}n+\frac{5}{11}$. Later Katz and Tao \cite{Katz&Tao2} developed the method further and got $(2 - \sqrt{2}) (n-4)$, which is the best estimate known at the moment for $n \ge 5$.

A natural approach to the problem is via a related maximal function, the \textit{Kakeya maximal function} $f^*_\delta$ (with width $\delta >0$). This is defined for $f \in L^1_{loc }(\mathbb{R}^{n})$ as
\begin{eqnarray*}
\begin{split}
&f^*_\delta: S^{n-1} \rightarrow [0, \infty],\\
f^*_\delta(e)= \sup_{a \in \mathbb{R}^{n}} & \frac{1}{\mathcal{L}^{n}(T^\delta_e(a))} \int_{T^\delta_e(a)} |f| d \mathcal{L}^{n},
\end{split}
\end{eqnarray*}
where $S^{n-1}$ is the unit sphere and $T^\delta_e(a)$ denotes the tube with center $a$, direction $e$, length 1 and radius $\delta$:
\begin{equation*}
T^\delta_e(a)= \{ p \in \mathbb{R}^{n}: |\langle p-a, e \rangle | \le 1/2, |p-a-\langle p-a, e \rangle e| \le \delta \}.
\end{equation*} The \textit{Kakeya maximal conjecture} states that the following inequality
\begin{equation}\label{Kakeyamaxfunction}
||f^*_\delta||_{L^n(S^{n-1})} \le C_{n,\epsilon} \delta^{- \epsilon} ||f||_{L^n(\mathbb{R}^n)}
\end{equation}
holds for all $\epsilon>0$, $0 < \delta <1$ and $f \in L^n(\mathbb{R}^n)$. In particular, this conjecture implies the Kakeya conjecture. %referenza 

We will show that, as in the case of the Euclidean Hausdorff dimension, estimates of type \eqref{Kakeyamaxfunction} for $L^p$ functions imply lower bounds for the Heisenberg Hausdorff dimension of Besicovitch sets. The proof we present holds only for bounded sets because of the structure of the Heisenberg group and in particular the fact that the directions of horizontal segments can get arbitrarily close to the vertical axis in unbounded sets. 

We will use the following notation. The Lebesgue measure in $\mathbb{R}^n$ is denoted by $\mathcal{L}^n$ and the $s$-dimensional (euclidean) Hausdorff measure by $\mathcal{H}^s$, $s\ge0$. We recall that this is defined for any $A \subset\mathbb{R}^n$ by
\begin{equation*}
\mathcal{H}^s(A)= \lim_{\delta \rightarrow 0}\mathcal{H}_\delta^s(A),
\end{equation*}
where for $\delta >0$
\begin{equation*}
\mathcal{H}_\delta^s(A) =\inf \left\{ \sum_i d(E_i)^s : A \subset \bigcup_i E_i, d(E_i) < \delta \right\}.
\end{equation*}
Here $d(E_i)$ denotes the diameter of $E_i$ with respect to the euclidean metric.

We denote by $\sigma^{n-1}$ the surface measure on $S^{n-1}$ and by $B_E(p,r)$ the Euclidean ball with center $p$ and radius $r$. The notation $C_a$ means that the constant $C$ depends on $a$.

The Heisenberg group $\mathbb{H}^n$ is $\mathbb{R}^{2n+1}$, where we denote the points by
\begin{equation*}
p=(x,y,t)=(x_1, \dots, x_n, y_1, \dots, y_n, t)
\end{equation*}
with $x_j$, $y_j$, $t \in  \mathbb{R}$ for every $j=1, \dots, n$, with the composition law given by
\begin{equation*}
(x,y,t) \ast (x',y',t')= (x+x', y+y', t+t'+2( \langle y, x' \rangle - \langle x, y' \rangle)),
\end{equation*}
where $\langle \cdot, \cdot \rangle$ is the usual inner product in $\mathbb{R}^n$. The inverse of $p$ is $p^{-1}=(-x,-y,-t)$.

We will work with the \textit{Heisenberg metric} on $\mathbb{H}^n$, also known as the \textit{Kor\'anyi metric}. This is the left invariant metric given by
\begin{equation*}
d_H(p,q)= ||q^{-1} \ast p ||_H,
\end{equation*}
where $|| \cdot ||_H$ is the homogeneous norm defined by
\begin{equation*}
||p||_H= ((||x||^2+||y||^2)^2 +t^2)^{\frac{1}{4}}.
\end{equation*}

Here $||\cdot||$ denotes the euclidean norm. We denote by $\tau_p$ the left translation by $p$, i.e. $\tau_p(q)= p*q$. Horizontal lines in $\mathbb{H}^n$ are either lines through the origin in the $xy$-hyperplane or left translations of them.

The Heisenberg ball centred at $p$ with radius $r$ is the set
\begin{equation*}
B_H(p,r)= \{ q \in \mathbb{H}^n : d_H(p,q) <r \}.
\end{equation*}

We will estimate the Heisenberg Hausdorff dimension, which means that we consider the $s$-dimensional Hausdorff measure $\mathcal{H}^s_H$ defined with respect to the Heisenberg metric. The Heisenberg Hausdorff dimension of a set $A \subset \mathbb{H}^n$ is then defined in the usual way as
\begin{equation*}
\dim_H A= \inf \{ s: \mathcal{H}^s_H(A)=0 \} = \sup \{ s: \mathcal{H}^s_H(A)= \infty \}.
\end{equation*}
The Heisenberg Hausdorff dimension is greater or equal to the Euclidean one; for example, the Heisenberg Hausdorff dimension of $\mathbb{H}^n$ is $2n+2$.

A \textit{Kakeya set} in $\mathbb{H}^{n}$ is a Borel set $B \subset \mathbb{H}^{n} $ having zero Lebesgue measure ($\mathcal{L}^{2n+1}(B)=0$) and containing a line segment of unit length in every direction. This means that for every $e \in S^{2n}$ there exists $b \in \mathbb{R}^{2n+1}$ such that $\{ te+b: t \in [0,1] \} \subset B$.

The \textit{Kakeya maximal function} $f^*_\delta$ of $f \in L^1_{loc }(\mathbb{R}^{2n+1})$ is defined as above.

We will prove the following result.

\begin{theorem}\label{Th1}
Let $1<p<2n+1$, $\beta>0$ such that $2n+2-\beta p >0$. If
\begin{equation}\label{kn}
||f^*_\delta||_{L^p(S^{2n})} \le C_{n,p,\beta} \delta^{-\beta} ||f||_p \quad \mbox{for all} \quad f \in L^p(\mathbb{R}^{2n+1}),
\end{equation}
then the Heisenberg Hausdorff dimension of every bounded Besicovitch set in $\mathbb{H}^n$ is at least $2n+2-\beta p$. In particular, if \eqref{kn} holds for some $p$, $1 \le p < \infty$, for all $\beta >0$, then the Heisenberg Hausdorff dimension of every bounded Besicovitch set in $\mathbb{H}^n$ is 2n+2. 
\end{theorem}

We recall that the analogous result in the case of Euclidean Hausdorff dimension gives the lower bound $2n+1 - \beta p$ (\cite{Mattila}, Theorem 22.9). Note that this shows in particular that Kakeya maximal conjecture implies Kakeya conjecture.

Using then inequalities of type \eqref{kn} proved by Wolff \cite{MR1363209} and Katz and Tao \cite{Katz&Tao2} for some values of $p$ and $\beta$, we will get as a corollary some lower bounds for the Heisenberg Hausdorff dimension of bounded Besicovitch sets.

\section{Bounds derived from estimates of the maximal function}

For the proof of Theorem \ref{Th1} we will use the following lemma, which states how close to the vertical axis horizontal segments can get in a bounded set.  

\begin{lem}\label{anglen}
Any horizontal segment contained in a bounded set $ \subset B_E(0,R) \subset \mathbb{H}^n$ forms an angle with the $xy$-hyperplane which is at most $ \arccos \frac{1}{\sqrt{1+4R^2}} $.
\end{lem}

\begin{proof}
Let $e=(e^1,e^2,0)=(e_1, \dots, e_n, e_{n+1}, \dots, e_{2n},0) \in S^{2n}$ be a direction in the $xy$-hyperplane. A horizontal line is either a line in the $xy$-hyperplane through the origin, which has the form \begin{equation*}
\gamma(s)=(se^1,se^2,0), \qquad s \in \mathbb{R},
\end{equation*}
or a left translation of it by $p=(x,y,t) \in \mathbb{H}^n$:
\begin{equation*}
\tau_p(\gamma(s))= (se^1+x, se^2+y, t+ 2s(\langle y, e^1 \rangle - \langle x, e^2 \rangle)), \qquad s \in \mathbb{R}.
\end{equation*}

We now want to find the angle between such lines and the $xy$-hyperplane. In the first case, of course, the angle is 0.

The direction of $\tau_p(\gamma(s))$ is given by the vector $u=(e^1,e^2, 2(\langle y, e^1 \rangle - \langle x, e^2 \rangle))$, whose projection on the $xy$-hyperplane is the vector $e=(e^1,e^2,0)$. The angle between them is given by
\begin{equation*}
\theta= \arccos \frac{\langle u, e \rangle}{||u|| ||e||}= \arccos \frac{1}{\sqrt{1+4(\langle y, e^1 \rangle - \langle x, e^2 \rangle)^2}}.
\end{equation*}

Since we are interested in segments contained in $B_E(0,R)$, we consider translations by points $p=(x,y,t) \in B_E(0,R)$, that is $||(x,y,t)|| < R$. Then we try to estimate $(\langle y, e^1 \rangle - \langle x, e^2 \rangle)^2$. We can write
\begin{equation*}
(\langle y, e^1 \rangle - \langle x, e^2 \rangle)^2= |\langle (e^1, -e^2), (y,x) \rangle_{\mathbb{R}^{2n}} |^2 \le ||(e^1,-e^2)||^2_{\mathbb{R}^{2n}} ||(y,x)||^2_{\mathbb{R}^{2n}},
\end{equation*} 
where the last inequality follows from Cauchy-Schwarz inequality. Since $||(e^1,-e^2)||_{\mathbb{R}^{2n}}=1$, we have
\begin{equation}\label{aybxRn}
(\langle y, e^1 \rangle - \langle x, e^2 \rangle)^2 \le ||(y,x)||^2_{\mathbb{R}^{2n}} < R^2-t^2 \le R^2.
\end{equation}

Thus 
\begin{equation*}
\theta \le \arccos \frac{1}{\sqrt{1+4R^2}}.
\end{equation*}
\end{proof}

The proof of Theorem \ref{Th1} proceeds in the same way as in the case of Euclidean Hausdorff dimension, see Theorem 22.9 in \cite{Mattila}.

\begin{proof} (Theorem \ref{Th1})\\
Let $B \subset \mathbb{H}^n$ be a bounded Besicovitch set, $B \subset B_E(0,R)$ for some $R>0$. Let $0<\alpha < 2n+2-\beta p$ and for $j \in \mathbb{N}$, $j \ge1$, let $B_j= B_H(p_j, r_j)$ be Heisenberg balls such that $p_j \in B_E(0,R)$, $B \subset \bigcup_j B_j$, and $r_j \le 2^{-M}$, where $M=\left[\frac{\log(3/\tan \theta_R)}{\log 2} \right] +1$ and $\theta_R= \arccos \frac{1}{\sqrt{1+4(R+1)^2}} - \arccos \frac{1}{\sqrt{1+4R^2}} >0$. Here $[\quad]$ denotes the integer part.

Let S be the set of  $e \in S^{2n}$ such that the angle between $e$ and the $xy$-hyperplane is bigger than $\theta$, where $\theta =\arccos \frac{1}{\sqrt{1+4(R+1)^2}} > \arccos \frac{1}{\sqrt{1+4R^2}}$.
Then, by Lemma \ref{anglen}, S does not contain any horizontal direction.

Since $B$ is a Besicovitch set, it contains a unit segment parallel to $e$ for every $e \in S$. Denote by $I_e \subset B$ one of these segments. For $k=M+1,\dots$, let
\begin{equation*}
J_k = \{j: 2^{-k} \le r_j < 2^{1-k} \},
\end{equation*}
and
\begin{equation*}
S_k= \left\{e \in S : \mathcal{H}^1\left(I_e \cap \bigcup_{j \in J_k} B_j \right) \ge \frac{1}{2k^2} \right\}.
\end{equation*}

Then we have
\begin{equation*}
 \bigcup_{k=M}^\infty S_k =S.
\end{equation*}

Indeed, if there were some $e \in S \setminus \bigcup_k S_k$, then we would have $\mathcal{H}^1(I_e \cap \bigcup_{j \in J_k} B_j ) < \frac{1}{2k^2}$ for all  $k \ge M$. Since $\sum_{k=M}^\infty \frac{1}{2k^2} < \sum_{k=1}^\infty \frac{1}{2k^2}<1$, this would imply
\begin{equation*}
\sum_{k=M}^\infty \mathcal{H}^1\left(I_e \cap \bigcup_{j \in J_k} B_j \right) < \sum_{k=M}^\infty \frac{1}{2k^2} <1,
\end{equation*}
which is impossible because
\begin{equation*}
\sum_{k=M}^\infty \mathcal{H}^1\left(I_e \cap \bigcup_{j \in J_k} B_j \right) \ge \mathcal{H}^1 (I_e) \ge 1.
\end{equation*}

Let \begin{equation*}
F_k= \bigcup_{j \in J_k} B_H (p_j, R_j) \quad \mbox{and} \quad f= \chi_{F_k},
\end{equation*}

where $R_j=2 \left( 1+ \frac{1}{\sin \theta_R}\right) r_j$.

We claim that for $e \in S_k$
\begin{equation}\label{intFk}
\mathcal{L}^{2n+1} (T \cap F_k ) \ge C_{R,n}  \frac{1}{k^2} \mathcal{L}^{2n+1}(T), 
\end{equation}
where $T=T^{2^{-k}}_e(b)$ with $b \in \mathbb{H}_n$ that will be specified later.

We first show how this concludes the proof and then we prove \eqref{intFk}.

By \eqref{intFk}, we have that, for $e \in S_k$, $f^*_{2^{-k}}(e) \ge C_{R,n} \frac{1}{k^2}$, which implies
\begin{equation*}
||f^*_{2^{-k}}||_p \ge C_{R,n} \frac{1}{k^2} \sigma^{2n}(S_k)^{1/p}.
\end{equation*}

Observe that $||f||_p = (\mathcal{L}^{2n+1}(F_k))^{1/p} \le (C_{R,n} \# J_k 2^{(1-k)(2n+2)})^{1/p}$, where $2^{(1-k)(2n+2)}$ is essentially the volume of a Heisenberg ball of radius $2^{1-k}$. Hence by the assumption \eqref{kn} we have
\begin{equation*}
||f^*_{2^{-k}}||_p \le C_{n,p,\beta} 2^{k\beta} ||f||_p \le  C_{n,p,\beta} 2^{k\beta} (\#J_k 2^{(1-k)(2n+2)})^{1/p}.
\end{equation*}

Combining these two inequalities, we get
\begin{eqnarray*}
\begin{split}
\sigma^{2n}(S_k) &\le C_{n,p, \beta,R} k^{2p} 2^{k \beta p} 2^{-(2n+2)k} \# J_k=C_{n,p, \beta,R} k^{2p} 2^{-k(2n+2-\beta p)} \# J_k \le \\ &\le C_{n,p, \beta,R} 2^{-k \alpha} \#J_k.
\end{split}
\end{eqnarray*}

Hence it follows
\begin{equation*}
\sum_{j=1}^\infty r_j^\alpha \ge \sum_{k=M}^\infty \# J_k 2^{-k \alpha} \ge C_{n,p, \beta,R} \sum_{k=M}^\infty \sigma^{2n}(S_k) \ge C_{n,p, \beta,R} \sigma^{2n}(S),
\end{equation*}
which concludes the proof.

We now prove \eqref{intFk}, which is trivial in the euclidean case.

First we observe that if $p \in B_H(p_j,r_j) \cap I_e$, then for any other $q \in B_H(p_j,r_j)$ we have $q \in B_H(p,2r_j)$. Hence, taking $\tilde{B}_j=B_H(p_j,2r_j)$, we can assume that $p_j \in I_e$. Then we still have for $e \in S_k$
\begin{equation*}
\mathcal{H}^1\left(I_e \cap \bigcup_{j \in J_k} \tilde{B}_j \right) \ge \frac{1}{2k^2}.
\end{equation*}

We will use the following notation: $I_e$ is given by the points $\gamma(u)=( ue_1+c_1, ue_2+c_2, ue_3 + c_3)$, with $e=(e_1,e_2,e_3) \in S$, $c=(c_1,c_2,c_3) \in \mathbb{H}^n$, $u \in [0,1]$, and $p_j=\gamma(\bar{u})=(\bar{x}, \bar{y}, \bar{t})$.

By definition, given some $j$, the Heisenberg ball $B_H(p_j,2r_j)$ is the set of points $q=(x,y,t) \in \mathbb{H}^n$ such that
\begin{equation*}
d_H(p_j,q)= ||q^{-1} \ast p_j||_H= ((||\bar{x}-x||^2 + ||\bar{y}-y||^2 )^2 + (\bar{t}-t + 2 (\langle x, \bar{y} \rangle - \langle y, \bar{x} \rangle)^2)^{1/4} <2 r_j. 
\end{equation*}

If we let $H_j= \{(x,y,t) \in \mathbb{H}^n : \bar{t}-t + 2 (\langle x, \bar{y} \rangle - \langle y, \bar{x} \rangle )=0 \}$, then 
\begin{equation*}
B_H(p_j,2r_j) \cap H_j= \{ (x,y,t) \in \mathbb{H}^n: (||\bar{x}-x||^2 + ||\bar{y}-y||^2 )^{1/2} < r_j, t= \bar{t} +  2 (\langle x, \bar{y} \rangle - \langle y, \bar{x} \rangle) \}.
\end{equation*}

\textbf{Step 1:} In case $p_j=(0,0,\bar{t})$, $H_j$ is parallel to the $xy$-hyperplane. If $p_j \ne (0,0,\bar{t})$, then we need to determine what angle $H_j$ makes with the $xy$-hyperplane.

A normal vector to $H_j$ is $n_j=(-2 \bar{y},  2 \bar{x}, 1)$, whereas the unit normal vector to the $xy$-hyperplane is $n=(0,0,1)$. Hence the angle between them (which is the angle between $H_j$ and the $xy$-hyperplane) is 
\begin{equation*}
\theta_j= \arccos \frac{\langle n_j, n \rangle}{||n_j|| ||n||}= \arccos \frac{1}{\sqrt{1+4(||\bar{y}||^2+|| \bar{x}||^2})}.
\end{equation*}

We have $||\bar{y}||^2+ ||\bar{x}||^2 < R^2$ because $p_j \in B_E(0,R)$. Hence
$\theta_j \le \arccos \frac{1}{\sqrt{1+4R^2}}$.

Since, for $e \in S_k$, the angle that $I_e$ makes with the $xy$-hyperplane is $\bar{\theta} > \arccos \frac{1}{\sqrt{1 + 4 (R+1)^2}}$, we have that the angle between $H_j$ and $I_e$ is
\begin{equation*}
\bar{\theta} - \theta_j \ge \theta_R.
\end{equation*}

\textbf{Step 2:} Now we claim that if $a \in I_e$ belongs to a ball $B_H(p_j,2 r_j)$, then all the segments with direction parallel to $H_j$, one endpoint in $a$ and length $\frac{2^{-k}}{\sin \theta_R}$ are contained in the ball $B_H(p_j, R_j)$. 

Let $a=(a_1,a_2,a_3) \in B_H(p_j,2 r_j)$. This means that
\begin{equation}\label{ainBrj}
(||\bar{x} - a_1||^2+ ||\bar{y}-a_2||^2)^2 + (\bar{t}-a_3 + 2 (\langle a_1, \bar{y} \rangle - \langle a_2, \bar{x} \rangle))^2 < (2r_j)^4.
\end{equation}

A point $q$ in any segment with direction parallel to $H_j$, one endpoint in $a$ and of length $ \frac{2^{-k}}{\sin \theta_R} $ is given by $q=\sigma^j_{e'}(s)$, with
\begin{equation}\label{segment}
\sigma^j_{e'}(s)=\left( se'_1+ a_1, s e'_2 + a_2, 2s (\langle e'_1, \bar{y} \rangle - \langle e'_2, \bar{x} \rangle) + a_3 \right),
\end{equation}
where $e'=(e'_1,e'_2)$ is in the unit sphere $S_P$ contained in the hyperplane $P-c$ and $P$ is the hyperplane orthogonal to $I_e$ passing through $c$ ($P-c$ is the translate of $P$ passing through the origin), $||e'_1||^2 + ||e'_2||^2 =1 $, and $0 \le s \le \frac{2^{-k}}{\sin \theta_R}  \le \frac{ r_j}{\sin \theta_R}$.

We want to show that $q \in B_H(p_j, R_j)$. This means that
\begin{eqnarray*}
\begin{split}
&(||\bar{x}-se'_1-a_1||^2+ ||\bar{y}- s e'_2-a_2||^2)^2 + \\ &+ \left(\bar{t}- 2s (\langle e'_1, \bar{y}, \rangle - \langle e'_2, \bar{x} \rangle) - a_3 + 2 ( \langle se'_1+a_1,\bar{y} \rangle - \langle s e'_2 + a_2, \bar{x} \rangle) \right)^2 < R_j^4, 
\end{split}
\end{eqnarray*}
that is 
\begin{eqnarray}\label{qinBRj}
\begin{split}
(||\bar{x}-se'_1-a_1||^2+ ||\bar{y}- s e'_2-a_2||^2)^2 + (\bar{t} - a_3 + 2 (\langle a_1, \bar{y} \rangle - \langle a_2, \bar{x} \rangle))^2 < R_j^4.
\end{split}
\end{eqnarray}

By a direct calculation, the left-hand side equals
\begin{eqnarray*}
\begin{split}
& (||\bar{x} - a_1||^2+ ||\bar{y}-a_2||^2)^2 + (\bar{t}-a_3 + 2 (\langle a_1, \bar{y} \rangle - \langle a_2, \bar{x} \rangle))^2 + \\
&+ s^4
- 4 s^3 (\langle e'_1, \bar{x}-a_1 \rangle + \langle e'_2, \bar{y}- a_2 \rangle )+ \\
& + 4 s^2 (\langle e'_1, \bar{x}-a_1 \rangle + \langle e'_2, \bar{y}- a_2 \rangle )^2+ 2 s^2 (||\bar{x} - a_1||^2+ ||\bar{y}-a_2||^2) + \\
&- 4s (||\bar{x} - a_1||^2+ ||\bar{y}-a_2||^2)(\langle e'_1, \bar{x}-a_1 \rangle + \langle e'_2, \bar{y}- a_2 \rangle ).
\end{split}
\end{eqnarray*}

Using \eqref{ainBrj} and the inequalities
\begin{itemize}

\item $s \le \frac{r_j}{\sin \theta_R}$, 

\item $ (\langle e'_1, \bar{x}-a_1 \rangle + \langle e'_2, \bar{y}- a_2 \rangle )^2 = (\langle (e'_1, e'_2), (\bar{x}-a_1, \bar{y}-a_2 ) \rangle )^2 \le \\ \le ||(e'_1, e'_2)||^2 ||(\bar{x}-a_1, \bar{y}-a_2 )||^2 \le ||\bar{x}-a_1||^2 + ||\bar{y}- a_2||^2 < 4 r_j^2$,

\end{itemize}
we find the following upper bound for the left-hand side of \eqref{qinBRj}:
\begin{equation*}\label{lhs}
16\left(1+ \frac{4}{\sin \theta_R}+ \frac{6}{(\sin  \theta_R)^2} + \frac{4}{(\sin \theta_R)^3} + \frac{1}{(\sin \theta_R)^4} \right) r_j^4= \left( 1 + \frac{1}{\sin \theta_R} \right)^4(2r_j)^4= R_j^4.
\end{equation*} 

Hence the claim is proved.

\textbf{Step 3:} The segments $\sigma^j_{e'}(s)$ considered above are in general not parallel to each other when $p_j$ varies along $I_e$.  

Consider the segment $\sigma^j_{e'}(s)$ given by \eqref{segment}, where $\bar{x}= ue_1+c_1$ and $\bar{y}=ue_2+c_2$. Its direction is given by the vector 
\begin{eqnarray*}
\begin{split}
v(u)&=( e'_1, e'_2, 2(\langle e'_1, ue_2+c_2 \rangle - \langle e'_2, ue_1+c_1 \rangle))= \\&=( e'_1, e'_2, 2( u C_{e,e'} + C_{e',c})),
\end{split}
\end{eqnarray*}
where $C_{e,e'}= \langle e'_1,e_2 \rangle - \langle e'_2,e_1 \rangle$ and $C_{e',c}= \langle e'_1,c_2 \rangle - \langle e'_2,c_1 \rangle $.

The norm of $v(u)$ is then $||v(u)||= \sqrt{1+4(u C_{e,e'} + C_{e',c})^2}$. Let now 
\begin{eqnarray*}
f(u)= \langle \frac{v(u)}{||v(u)||}, e \rangle = \frac{\langle e'_1, e_1\rangle + \langle e'_2,e_2 \rangle + 2 e_3 ( u C_{e,e'} + C_{e',c})}{\sqrt{1+4(u C_{e,e'} + C_{e',c})^2}}.
\end{eqnarray*}

Then $\arccos f(u)$ is the angle between $I_e$ and the segment $\sigma^j_{e'}(s)$. Observe that if $C_{e,e'}= \langle e'_1,e_2 \rangle - \langle e'_2,e_1 \rangle=0$ then $f$ is constant, hence the segments $\sigma^j_{e'}(s)$ are parallel as $p_j$ moves along $I_e$ (this happens for example if $e_1=e_2=0$, i.e. $I_e$ is vertical).

The derivative of $f$ is given by
\begin{eqnarray*}
\frac{d}{du} f(u)= \frac{2 C_{e,e'} (e_3-2C_{e',c} (\langle e'_1, e_1\rangle + \langle e'_2,e_2 \rangle) - 2u C_{e,e'} (\langle e'_1, e_1\rangle + \langle e'_2,e_2 \rangle))}{(1+4(u C_{e,e'} + C_{e',c})^2)^{3/2}}.
\end{eqnarray*}

If $\langle e'_1, e_1\rangle + \langle e'_2,e_2 \rangle=0$, then the sign of $\frac{d}{du}f$ is constant (same sign as $C_{e,e'}$), hence $f$ is monotone. If $\langle e'_1, e_1\rangle + \langle e'_2,e_2 \rangle \ne 0$ and $C_{e,e'}\ne 0$, let
\begin{equation}
C_{e,e',c}= \frac{e_3-2C_{e',c} (\langle e'_1, e_1\rangle + \langle e'_2,e_2 \rangle)}{2C_{e,e'} (\langle e'_1, e_1\rangle + \langle e'_2,e_2 \rangle)}.
\end{equation}

If $C_{e,e',c} <0$ or $C_{e,e',c} >1$, then $f$ is monotone on $[0,1]$. If $0 \le  C_{e,e',c} \le 1$, then if $f$ is increasing (respectively decreasing) for $u \in [0, C_{e,e',c})$, it is decreasing (respectively increasing) for $u \in ( C_{e,e',c},1]$.

In case $f$ is monotone for all $u \in [0,1]$, let $I_{e,e'}=I_e$. Otherwise, let $I^1_e= \gamma([0,C_{e,e',c}])$ and $I^2_e= \gamma([C_{e,e',c},1])$. Since
\begin{equation*}
\mathcal{H}^1\left(I_e \cap \bigcup_{j \in J_k} \tilde{B}_j \right) \ge \frac{1}{2k^2},
\end{equation*}
there exists $i \in \{1,2\}$ such that 
\begin{equation*}
\mathcal{H}^1\left(I_e^{i} \cap \bigcup_{j \in J_k} \tilde{B}_j \right) \ge \frac{1}{4k^2}.
\end{equation*}
Let $I_{e,e'}=I_e^{i}$. 

Since $f$ is monotone on $I_{e,e'}$, either the segments $\sigma^j_{e'}(s)$ or the segments $\sigma^j_{-e'}(s)$ do not intersect when $p_j \in I_{e,e'}$. Hence there is a subset $S_P'$ of $S_P$ with $\sigma^{2n-1}(S_P') = \sigma^{2n-1}(S_P)/2$ and for every $e' \in S_P'$ there exists $I_{e,e'}$ such that the segments $\sigma^j_{e'}(s)$ do not intersect for $p_j \in I_{e,e'}$ and
\begin{equation*}
\mathcal{H}^1\left(I_{e,e'} \cap \bigcup_{j \in J_k} \tilde{B}_j \right) \ge \frac{1}{4k^2}.
\end{equation*}

Let $I_1= \gamma([0,1/3])$, $I_2= \gamma([1/3,2/3])$ and $I_3=\gamma([2/3,1])$. Then there exists $i_{e'} \in \{1,2,3\}$ such that
\begin{equation}\label{Ii}
\mathcal{H}^1\left((I_{e,e'}\cap I_{i_{e'}}) \cap \bigcup_{j \in J_k} \tilde{B}_j \right) \ge  \frac{1}{12k^2}.
\end{equation}

Let $S_1=\{ e' \in S_P' : i_{e'}=1\}$, $S_2=\{ e' \in S_P' : i_{e'}=2\}$ and $S_3=\{ e' \in S_P' : i_{e'}=3\}$. Then there exists $l \in \{1,2,3\}$ such that $\sigma^{2n-1}(S_l) \ge \frac{\sigma^{2n-1}(S_P')}{3}= \frac{\sigma^{2n-1}(S_P)}{6}$.

If $l=2$ then we let $T=T^{2^{-k}}_e(\gamma(1/2))$. In case $l=1$ or $3$, we take $T=T^{2^{-k}}_e(b)$, where $b$ is the mid-point of $I_1$ or $I_3$ respectively. 

For any $e' \in S_l$, consider the 2-dimensional plane $P'$ orthogonal to $P$ and containing $I_e$ and the segment joining $c$ and $e'+c$. The line containing $I_e$ divides $P'$ into two half-planes. Let $P''$ be the one that contains the segment joining $c$ and $e'+c$. Note that for $k \ge M$ (as we assumed) we have $\frac{2^{-k}}{\tan \theta_R} < \frac{1}{3} $. Thus for $p_j \in I_l$ the segments $\sigma_{e'}^j(s)$ intersect any segment $I$ parallel to $I_e$ contained in $T \cap P''$.

Moreover, by step 2 all the segments $\sigma^j_{e'}(s)$ are contained in $F_k$ and by step 3 they do not intersect for $p_j \in I_{e,e'}$. Hence by \eqref{Ii} if $I$ is any segment parallel to $I_e$, contained in $P'' \cap T$, we have
\begin{equation}\label{10}
\mathcal{H}^1\left(I \cap F_k \right) \ge C_{R,n} \frac{1}{k^2}.
\end{equation} 

Hence $\mathcal{L}^{2}(P''\cap T \cap F_k) \ge C_{R,n} \frac{1}{k^2}\mathcal{L}^{2}(P'' \cap T)$. This holds for all such half-planes $P''$ containing the segment from $c$ to $e'+c$ with $e' \in S_l$. To get \eqref{intFk}, we will first integrate over 2-dimensional planes $P''$ and then over $S_l$. 

If $f$ is a Lebesgue measurable function in $\mathbb{R}^N$, then by polar integration in translates of the $x_1, \dots , x_{N-1}$-hyperplane along the $x_N$-axis we obtain
\begin{equation*}
\int_{\mathbb{R}^N} f(x) dx = \int_{-\infty}^\infty \left( \int_0^\infty r^{N-2} \int_{S^{N-2}} f(ry , x_N ) d \sigma^{N-2}(y) dr \right) dx_N ,
\end{equation*}
where $y \in S^{N-2}$ and we have written $x=(x',x_N) \in \mathbb{R}^N$ with $ x'\in \mathbb{R}^{N-1}$. Changing the order of integration, we get
\begin{equation*}
\int_{\mathbb{R}^N} f(x) dx =\int_{S^{N-2}} \left( \int_0^\infty r^{N-2} \int_{-\infty}^\infty  f(ry , x_N ) d x_N dr \right) d \sigma^{N-2}(y),
\end{equation*}
where we first integrate over 2-dimensional planes containing the $x_N$-axis and then over the unit sphere in the $x_1, \dots , x_{N-1}$-hyperplane. 

If we apply this formula to the case when $N=2n+1$, the $x_N$-axis is the line containing $I_e$ and the $x_1, \dots , x_{N-1}$-hyperplane is $P-c$, then \eqref{intFk} follows from \eqref{10}.

\end{proof}

We recall Wolff's result, which in case of a Besicovitch set in $\mathbb{R}^{2n+1}$ gives the lower bound $\frac{2n+3}{2}$ for its Euclidean Hausdorff dimension.
%mettere referenza!!!

\begin{theorem}
Let $0 < \delta <1$. Then for $f \in L^{\frac{2n+3}{2}}(\mathbb{R}^{2n+1})$,
\begin{equation}\label{Wolff}
||f^*_\delta ||_{L^{\frac{2n+3}{2}}(S^{2n})} \le C_{n,\epsilon} \delta^{\frac{1-2n}{3+2n}- \epsilon} ||f||_{L^{\frac{2n+3}{2}}(\mathbb{R}^{2n+1})}
\end{equation}
for every $\epsilon >0$.
\end{theorem} 

Later Katz and Tao improved this result for $2n+1 \ge 9$, proving the following. %referenza!!!

\begin{theorem}
Let $0 < \delta <1$. Then for $f \in L^{\frac{8n+7}{7}}(\mathbb{R}^{2n+1})$,
\begin{equation}\label{KT}
||f^*_\delta ||_{L^{\frac{8n+7}{7}}(S^{2n})} \le C_{n,\epsilon} \delta^{-\frac{6n}{8n+7}- \epsilon} ||f||_{L^{\frac{8n+7}{7}}(\mathbb{R}^{2n+1})}
\end{equation}
for every $\epsilon >0$.
\end{theorem}

These results give the following estimates for the Heisenberg Hausdorff dimension of bounded Besicovitch sets.

\begin{cor}
Every bounded Besicovitch set in $\mathbb{H}^n$ has Heisenberg Hausdorff dimension at least $\frac{2n+5}{2}$ for $n \le 3$ and $\frac{8n+14}{7}$ for $n \ge 4$.
\end{cor} 

\begin{proof}
The estimate \eqref{Wolff} corresponds to \eqref{kn} with $p=\frac{2n+3}{2}$ and $\beta = \frac{2n-1}{2n+3}+ \epsilon$. Hence it gives the lower bound $2n+2-\beta p= \frac{2n+5}{2}- \epsilon \frac{2n+3}{2}$. Taking the limit when $\epsilon \rightarrow 0$, we get the claim.   

Similarly, from \eqref{KT} we get the lower bound $\frac{8n+14}{7}$.
\end{proof}

\begin{rem}
Closely related to Kakeya sets are Nikodym sets. A \textit{Nikodym set} $E \subset \mathbb{R}^n$ is a Borel set such that $\mathcal{L}^n(E)=0$ and for every $x \in \mathbb{R}^n$ there is a line $L$ through $x$ such that $E \cap L$ contains a unit line segment. The Nikodym conjecture states that Nikodym sets have full Hausdorff dimension and it is implied by Kakeya conjecture.

The \textit{Nikodym maximal function} $f^{**}_\delta$ of $f \in L^1_{loc}(\mathbb{R}^n)$ is defined as
\begin{equation*}
f^{**}_\delta(x)=\sup_{x \in T} \frac{1}{\mathcal{L}^n(T)} \int_T |f| d \mathcal{L}^n, \quad x \in \mathbb{R}^n,
\end{equation*}
where the supremum is taken over all tubes $T=T^\delta_e(a)$ that contain $x$. The Nikodym maximal conjecture asserts that the following inequality 
\begin{equation*}
||f^{**}_\delta||_{L^n(\mathbb{R}^n)} \le C_{n,\epsilon} \delta^{- \epsilon} ||f||_{L^n(\mathbb{R}^n)}
\end{equation*}
holds for all $\epsilon>0$, $0 < \delta <1$ and $f \in L^n(\mathbb{R}^n)$. In \cite{Tao} Tao proved that Nikodym maximal conjecture and Kakeya maximal conjecture are equivalent.

The proof of Theorem \ref{Th1} can be easily modified (by taking the unit ball instead of the unit sphere) to treat the case of the Nikodym maximal function. The theorem still holds if we replace $f^*_\delta$ by $f^{**}_\delta$ and Besicovitch sets by Nikodym sets. In particular, Nikodym maximal conjecture implies Nikodym conjecture and the same results obtained above hold also for Nikodym sets.
\end{rem}

\textit{Acknowledgements.} The author would like to thank Pertti Mattila for constant support and helpful comments.

\bibliography{Hei_Hausdorff}{}
\bibliographystyle{plain}

Department of Mathematics and Statistics, P.O. Box 68 (Gustaf H\"allstr\"omin katu 2b), FI-00014 University of Helsinki, Finland\\
E-mail: laura.venieri@helsinki.fi

\end{document}